\documentclass[12pt]{article}
\usepackage{amsmath,amsthm,amsfonts,amssymb,amscd}
\usepackage[usenames]{color}
\usepackage[dvips]{graphicx}
\usepackage{epsfig}
\def\phi{\varphi}

\begin{document}

\newtheorem{lem}{Lemma}
\newtheorem{predl}{Proposition}
\newtheorem{theorem}{Theorem}
\newtheorem{defin}{Definition}
\newtheorem{zam}{Remark}

\bigskip
\centerline{\bf TESTS OF EXPONENTIALITY BASED ON YANEV-CHAKRABORTY}
\centerline{\bf CHARACTERIZATION, AND THEIR EFFICIENCY}

\bigskip

\centerline{Volkova K. Y.\footnote{Research supported by grant RFBR No. 13-01-00172, grant NSh No. 2504.2014.1 and by SPbGU grant No. 6.38.672.2013}}

\bigskip
\centerline{Saint-Petersburg State University, Russia}
 \section{ Introduction}
\bigskip

In this paper we develop goodness-of-fit tests for exponentiality using a characteriza\-tion based on property of order statistics.
The problem formulation is as follows: let $X_1,X_2,\ldots,X_n$ be i.i.d. observations having the continuous df $F$. Consider testing of composite hypothesis
of exponentiality $ H_0: \mbox{$F \in {\cal{E(\lambda)}}$},
$ where ${\cal {E(\lambda)}}$ denotes the class of exponential
distributions with the density
$f(x)=\lambda e^{-\lambda x}, x \geq 0,$ here $\lambda>0$ is some unknown parameter.

There exists considerable literature on the problem of testing exponentiality. Such tests are constructed by different techniques, see books and reviews
 \cite{Ahsan},  \cite{A90}, \cite{BB95}, \cite{DY}, \cite{HM02}, \cite{NJ02}. Some tests of exponentiality  are based on the loss-of-memory pro\-perty, see \cite{AA},
 \cite{Ang}, \cite{Kou}, and several tests
use other characterizations of exponentiality \cite{Litv2}, \cite{Nou}, \cite{BarHen}, \cite{Kou2}, \cite{RT}, \cite{Niknik}, \cite{Swa}, \cite{NikVol}, \cite{Rank}.

We use the idea of testing statistical hypotheses using the characterizations by the property of equidistribution, and construct test statistics by means of so-called $V$- and $U$-empirical df's, see \cite{Jan}, \cite{Kor}.
Let us explain this method.

Suppose that the df $F$ belongs to the class of distributions $\cal F$,
if the corresponding density $f$ has derivatives of all orders in the neighbourhood of zero.

Arnold and Villasenor conjectured in \cite{Ar}, and Yanev and Chakraborty  proved recently in \cite{Yanev}, see also \cite{chakra}, that the following property characterizes the exponential law within the class $\cal F:$

\emph{Let $X_1,\ldots,X_n$ be non-negative i.i.d. rv's with df $F$ from class $\cal F.$
 Then the statistics $\max(X_1, X_2, X_3)$ and
$\max(X_1, X_2)+ \frac13 X_3$ are identically distributed if and only if the df $F$ is exponential.}

Consider the usual empirical df $F_n(t)=n^{-1}\sum_{i=1}^n\textbf{1}\{X_i<t\}, t \in R^1,$
based on the observations $X_1,\dots,X_n.$ According to our characterization we construct
for $t\geq 0$ the $V$-empirical df's by the formulae
\begin{align*}
H_n(t)=\frac{1}{n^{3}}\sum_{i_1,i_2,i_3
=1}^{n}\textbf{1}\{\max(X_{i_1}, X_{i_2},  X_{i_3})<t&\},\quad
t\geq 0,\\
G_n(t)=\frac{1}{3n^{3}}\sum_{i_1,i_2,i_3
=1}^{n}[\textbf{1}\{\max(X_{i_1}, X_{i_2})+ \frac{X_{i_3}}{3}&<t\}+\\+
\textbf{1}\{\max(X_{i_2}, X_{i_3})+ \frac{X_{i_1}}{3}<t\}+
\textbf{1}\{\max&(X_{i_3}, X_{i_1})+ \frac{X_{i_2}}{3}<t\}], \quad t\geq 0.
\end{align*}

It is known that the properties of $V$- and $U$-empirical df's are similar to
the properties of usual empirical df's, see \cite{HJS}, \cite{Jan}. Hence for large $n$ the df's $H_n$ and  $G_n$ should be close under $H_0,$ and we can
measure their closeness by using some test statistics.

We suggest two scale-invariant statistics
\begin{align}
I_n&=\int_{0}^{\infty} \left(H_n(t)-G_n(t)\right)dF_n(t),\label{I_n_Ar}\\
D_n&=\sup_{t \geq 0}\mid H_n(t)-G_n(t)\mid\label{D_n_Ar},
\end{align}
assuming that their large values are critical.

We discuss their limiting distributions under the null hypothesis and calculate
their efficiencies against common alternatives from the class $\cal F.$
The statistic $D_n$ has the non-normal limiting distribution,
hence we use the notion of local exact Bahadur efficiency (BE)
\cite{Bahadur}, \cite{Nik}, because the Pitman approach to efficiency is not applicable.
However, it is known that the local BE and the limiting Pitman efficiency
usually coincide, see  \cite{Wie}, \cite{Nik}.

The  large deviation asymptotics is the key tool for the evaluation of the
exact BE, and we address this question using the results of \cite{nikiponi}.
Finally, we study the conditions of local optimality of our tests and describe the
"most favorable" alternatives for them.

\bigskip

 \section{ Integral statistic  $I_{n}$}

\bigskip

Without loss of generalization we can assume that  $\lambda=1$. The statistic
$I_{n}$ is asymptotically equivalent to the $V$-statistic of degree  $4$  with
the centered kernel $\Psi(X_1, X_2,X_3, X_4)$ given by
\begin{gather*}
\Psi(X_{1}, X_{2},X_{3}, X_{4})=\frac14 \sum_{\pi(i_1, \ldots,
i_{4})} \textbf{1}\{\max(X_{i_1}, X_{i_2}, X_{i_3})<X_{i_{4}}\}-\\-
\frac{1}{24}\sum_{\pi(i_1, \ldots, i_{4})}
\textbf{1}\{\max(X_{i_1}, X_{i_2})+ \frac{X_{i_3}}{3}<X_{i_{4}}\},
\end{gather*}
where $\pi(i_1, \ldots,
i_{4})$ means all permutations of different indices  from $\{i_1, \ldots, i_4\}.$

\begin{theorem}
Under null hypothesis as $n \rightarrow \infty$ the statistic $I_{n}$ is asymptotically normal
with asymptotic variance given by
$$\sqrt{n}I_n \stackrel{d}{\longrightarrow}{\cal
{N}}(0,\frac{23}{10920}).$$
\end{theorem}

\begin{proof}Let $X_1,\ldots, X_{4}$ be independent standard exponential rv's. It is well-known that non-degenerate $V$- and $U$ -statistics are asymptotically normal, see \cite{Hoeffding},
\cite{Kor}. To prove that
 the kernel $\Psi(X_1, X_2,X_3,
X_4)$ is non-degenerate, let calculate its projection $\psi(s).$ For fixed $X_{4}=s$ we have:
\begin{align*}
\psi(s)=&E(\Psi(X_1, X_2,X_3, X_{4})\mid
X_{4}=s)=\\=&\frac{1}{4}P(\max(X_1,
X_2,X_3)<s)+\frac{3}{4}P(\max(s,X_2,X_3)<X_1)-\\-&
\frac{1}{4}P( \max(X_1, X_2)+ \frac{X_3}{3}<s)-\frac{1}{4}P( \max(X_1, X_2)+ \frac{s}{3}< X_3)-\\
- &\frac{1}{2}P( \max(s,X_1)+ \frac{X_2}{3}< X_3).
\end{align*}

It follows from the above characterization that the first and the third probability both are equal to:
$$P(\max(X_1,
X_2,X_3)<s)=\frac{1}{4}P( \max(X_1, X_2)+ \frac{X_3}{3}<s)=(1-e^{-s})^3.$$

The second term can be evaluated as follows:
\begin{multline*}
P(\max(s,X_2,X_3)<X_1)=P(s<X_1,
s>X_2, s>X_3)+\\+2 P(X_2<X_1, X_2>s,
X_2>X_3)=\\=(1-F(s))F(s)^{2}+
\left[\frac{2}{3}F(s)^3-F(s)^{2}+\frac13\right]=\frac{1-F(s)^3}{3}.
\end{multline*}

It remains to calculate two last terms, the calculations give us:
\begin{align*}
P( \max(X_1, X_2)+ \frac{s}{3}< X_3)=&\int_{s/3}^{\infty}F^2(x-s/3)dF_x= \frac13 e^{-s/3},\\
P( \max(s,X_1)+ \frac{X_2}{3}< X_3)=&
P(s<X_1, X_1+\frac{X_2}{3} <X_3)+\\+&P(s>X_1, s+\frac{X_2}{3} <X_3)=\frac34 e^{-s}- \frac38 e^{-2s}.
\end{align*}

Hence we get the final expression for the projection of the kernel $\Psi:$
\begin{gather}\label{psi}
\psi(s)=
\frac{3}{8}e^{-s}-\frac{9}{16}e^{-2s}+\frac{1}{4}e^{-3s}-\frac{1}{12}e^{-s/3}.
\end{gather}

The variance of the projection $\Delta^2 = E\psi^2(X_1)$ under $H_{0}$
is  given by
\begin{gather*}
\Delta^2 = \int_{0}^{\infty} \psi^2 (s) e^{-s}ds =\frac{23}{174720}\approx
0.0001316.
\end{gather*}

Therefore the kernel $\Psi $ is non-degenerate. Due to Hoeffding's  theorem on asymptotic normality of $V$- and $U$-statistics, see again \cite{Hoeffding}, \cite{Kor}, we get
the statement of the theorem.
\end{proof}

\section{Large deviations and local efficiency of $I_n$}

Now we shall evaluate the large deviation asymptotics of the sequence of statistics (\ref{I_n_Ar}) under $H_0.$ The kernel $\Psi$ is centered, bounded and non-degenerate.
Hence according to the
theorem on large deviations of such statistics from \cite{nikiponi}, see also \cite{anirban},
 \cite{Niki10}, we obtain the following result.
\begin{theorem}
For  $a>0$
$$
\lim_{n\to \infty} n^{-1} \ln P ( I_n >a) = - f(a),
$$
where the function $f$ is continuous for sufficiently small $a>0,$ and $$
f(a) \sim \frac{a^2}{32 \Delta^2_{\psi}} \sim   \frac{5460}{23}a^2, \, \mbox{as} \, a \to 0.
$$
\end{theorem}

Suppose that under the alternative $H_1$ the observations have the df $G(\cdot,\theta)$ and the density $g(\cdot,\theta), \ \theta \geq 0,$ such that
$G(\cdot, 0) \in {\cal {E(\lambda)}}.$
The measure of BE for any sequence $\{T_n\}$ of test statistics is the exact slope
$c_{T}(\theta)$ describing the rate of exponential decrease for the
attained level under the alternative df $G(\cdot,\theta).$ According to Bahadur theory  \cite{Bahadur}, \cite{Nik} the exact slopes may be found by
using the following Proposition.

\noindent {\bf Proposition}.\,{\it Suppose that the following two
conditions hold:
\[
\hspace*{-3.5cm} \mbox{a)}\qquad  T_n \
\stackrel{\mbox{\scriptsize $P_\theta$}}{\longrightarrow} \
b(\theta),\qquad \theta > 0,\nonumber \] where $-\infty <
b(\theta) < \infty$, and $\stackrel{\mbox{\scriptsize
$P_\theta$}}{\longrightarrow}$ denotes convergence in probability
under $G(\cdot\ ; \theta)$.
\[
\hspace*{-2cm} \mbox{b)} \qquad \lim_{n\to\infty} n^{-1} \ \ln \
P_{H_0} \left( T_n \ge t \ \right) \ = \ - h(t)\nonumber
\] for any $t$ in an open interval $I,$ on which $h$ is
continuous and $\{b(\theta), \: \theta > 0\}\subset I$. Then
$c_T(\theta) \ = \ 2 \ h(b(\theta)).$}

Note that the exact slopes always satisfy the inequality \cite{Bahadur}, \cite{Nik}
\begin{equation}
\label{Ragav}
c_T(\theta) \leq 2 K(\theta),\, \theta > 0,
\end{equation}
where $K(\theta)$ is the Kullback-Leibler "distance"\, between the alternative and the null-hypothesis $H_0.$ In our case $H_0$ is composite, hence for any
alternative density $g_j(x,\theta)$ one has
$$
K_j(\theta) = \inf_{\lambda>0} \int_0^{\infty} \ln [g_j(x,\theta) / \lambda \exp(-\lambda x) ] g_j(x,\theta) \ dx.
$$
This quantity can be easily calculated as $\theta \to 0$ for particular alternatives.
According to (\ref{Ragav}), the local BE of the sequence of statistics ${T_n}$ is defined as
$$
e^B (T) = \lim_{\theta \to 0} \frac{c_T(\theta)}{2K(\theta)}.
$$

Now we will give some examples of efficiency calculations. First consider the Makeham density
$$ g_1(x,\theta)=(1+\theta(1-e^{-x}))\exp(-x-\theta( e^{-x}-1+x)),\theta \geq 0, x
\geq 0,$$
and the corresponding df $G_1(x, \theta)$. According to the Law of Large Numbers
for $U$- and $V$-statistics \cite{Kor}, the limit in probability under $H_1$ is equal to
\begin{gather*}
b_1(\theta)=P_{\theta}(\max(X, Y, Z)<W)-P_{\theta}(\max(X, Y)+ \frac{Z}{3}<W).
\end{gather*}
It is easy to show (see also \cite{NiPe}) that
$$b_1(\theta) \sim 4\theta  \int_{0}^{\infty} \psi(s)h_1(s)ds,
$$
where $h_1(s)=\frac{\partial}{\partial\theta}g_1(s,\theta)\mid _{\theta=0}$ and $\psi(s)$ is the projection from (\ref{psi}).
Therefore for the Makeham alternative we have
\begin{gather*}
b_1(\theta) \sim 4\theta
\int_{0}^{\infty}(\frac{3}{8}e^{-s}-\frac{9}{16}e^{-2s}+\frac{1}{4}e^{-3s}-\frac{1}{12}e^{-s/3})
e^{-s}(2-2e^{-x}-x)ds =\\=\frac{3}{280}\theta, \quad \theta \to 0,
\end{gather*}
and the local exact slope of the sequence $I_n$ as $\theta \to 0$ admits the representation
$$c_1(\theta)=b^2_1(\theta)/(16\Delta^2) \sim 0.055\theta^2.$$

The Kullback-Leibler "distance" \, $K_1(\theta)$ between the Makeham distribution and the null-hypothesis $H_0$ satisfies   $K_1(\theta) \sim \frac{\theta^2}{24}, $ $\theta \to
0.$ Hence the local BE is equal to
$$e^B(I)=\lim_{\theta \to 0}\frac{c_1(\theta)}{2K_1(\theta)}\approx 0.654.$$

Consider the Weibull alternative with the density
$$ g_2(x,\theta)=(1+\theta)x^\theta \exp(-x^{1+\theta}),\theta \geq 0, x \geq
0,$$ and the corresponding df $G_2(x, \theta)$. After some calculations we have:
\begin{gather*}
b_2(\theta) \sim  (\frac{7}{16}\ln(3)-\frac{5}{8}\ln(2) )\theta\approx
0.047 \,\theta, \quad \theta \to 0.
\end{gather*}
The local exact slope admits the representation  $c_2(\theta) \sim 1.068\,\theta^2,$ $\theta
\to 0,$ while $K_2(\theta) \sim \pi^2\theta^2/12,$ $\theta \to 0.$  Consequently, the local efficiency of the test is  0.649.

The third is the Lehmann alternative with d.f.
$$ G_3(x,\theta)=F^{1+\theta}(x), \theta \geq 0, x\geq 0.$$
Omitting the calculations  similar to previous cases we get $
b_3(\theta)\sim (\frac23-\frac38 \ln{3}-\frac{\sqrt{3}}{24}\pi)\theta,$ $c_3(\theta)\sim 0.371\,\theta^2,$ $ \theta \to 0.$ It is easy to show that
$K_3(\theta) \sim \frac{\pi^2 (12-\pi^2)}{72} \theta^2,$ $\theta \to 0.$ Therefore the local BE is equal to 0.636.

The last alternative is the gamma-density
$$ g_4(x,\theta)=\frac{x^{\theta}}{\Gamma(\theta+1)}e^{-x}, \theta \geq 0, x\geq 0.$$
In this case we get  $ b_4(\theta)\sim(\frac12\ln(3)-\frac34\ln(2) )\theta \approx 0.029 \,\theta,$ $c_4(\theta)\sim 0.412\,\theta^2,
K_4(\theta) \sim
(\frac{\pi^2}{12}-\frac{1}{2})\theta^2, \, \theta \to 0.$ Hence the local BE is
equal to 0.638.

Next table gathers the values of local BE.

\begin{table}[!hhh]\centering
\caption{Local Bahadur efficiencies of the statistic $S_n$ under alternatives.}
\medskip
\begin{tabular}{|c|c|c|c|c|}
\hline
Alternative & Makeham & Weibull & Lehmann & Gamma \\
\hline
Efficiency & 0.654 & 0.649 &  0.636   & 0.638 \\
\hline
\end{tabular}
\end{table}
\bigskip

\section{Kolmogorov-type statistic $D_n$}

\bigskip

Now we consider the Kolmogorov type statistic (\ref{D_n_Ar}). Its indisputable merit is consistency against any alternative that follows directly from the characterization as such, while the integral statistic $I_n$ is not always consistent.

In our case for fixed  $t\geq 0$ the difference $H_n(t) - G_n(t)$
is a family of $V$-statistics with the kernels, depending on $t\geq 0:$
\begin{gather*}
\Xi(X,Y, Z;t)= \textbf{1}\{\max(X,Y, Z)<t\}-\frac13 \textbf{1}\{\max(X, Y)+ \frac{Z}{3}< t\}-\\
  -\frac13 \textbf{1}\{\max( Y, Z)+ \frac{X}{3}< t\}-\frac13 \textbf{1}\{\max(X, Z)+ \frac{Y}{3}< t\}.
\end{gather*}

The projection of this kernel $\xi(s;t):=E(\Xi(X, Y, Z; t)\mid X=s)$ for fixed
$t$ has the form:
\begin{equation*}
\xi(s;t)=  P\{\max(s,Y, Z)<t\}-\frac23 P\{\max(s, Y)+ \frac{Z}{3}< t\}-\frac13 P\{\max( Y, Z)+ \frac{s}{3}< t\}.
\end{equation*}

After standard computation we get:
\begin{align*}
P\{\max(s,Y, Z)<t\}=\textbf{1}\{s<t\}F^2(t)&,\\
P\{ \max(s,Y)+ \frac{Z}{3}< t\}=P(s<Y, \, Y&+\frac{Z}{3} <t)+P(s>Y, s+\frac{Z}{3} <t)=\\
=&\textbf{1}\{s<t\} \left(1-e^{3s-3t}+\frac32e^{2s-3t}-\frac32e^{-t}\right),\\
P\{\max( Y, Z)+ \frac{s}{3}< t\}=\textbf{1}\{s<3t\}F^2&(t-\frac{s}{3}).
\end{align*}

Combining the results obtained, we find that the projection
$\xi(s;t)$ for fixed $t$ is equal:
\begin{equation}
\xi(s;t)=  \textbf{1}\{s<t\}\left[\frac13-e^{-t}+e^{-2t}-e^{2s-3t}+\frac23e^{3s-3t}\right]  -\frac13 \textbf{1}\{s<3t\}(1-e^{t-s/3})^2.\label{xi}
\end{equation}
Now let find the variance function $\delta^2(t)= E\xi^2(X_1,t)$ of this projection under
$H_{0}.$ We have after some simple calculations:
\begin{multline*}
\delta^2(t)=\frac{1}{5}e^{-t}-\frac23 e^{-2t}+(\frac{26}{9}-\frac43 t)e^{-3t}
-\frac{43}{21}e^{-4t}+\\
+\frac{1}{10}e^{-5t}-\frac{4}{45}e^{-6t}-\frac{2}{7}e^{-5t/3}+\frac{1}{2}e^{-7t/3}+e^{-8t/3}
-\\-\frac85 e^{-10t/3}-2e^{-11t/3} +2e^{-13t/3}.
\end{multline*}
\begin{figure}[h!]
\begin{center}
\includegraphics[scale=0.35]{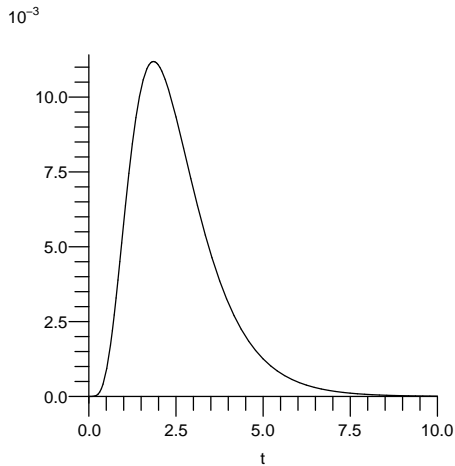}\caption{Plot of the function $\delta^2(t).$ }
\end{center}
\end{figure}

It is seen that our family of kernels $\Xi(X,Y,Z;t)$ is non-degenerate in the sense of \cite{Niki10} and $\delta^2=\sup_{t\geq 0}\delta^2(t) \approx 0.01119.$
This value will be important in the sequel when calculating the large deviation asymptotics.

The limiting distribution of the statistic $D_n$ is unknown. Using the methods of \cite{Silv}, one can show that the
$U$-empirical process
$$\eta_n(t) =\sqrt{n} \left(H_n(t) - G_n(t)\right), \ t\geq 0,
$$
weakly converges in $D(0,\infty)$ as $n \to \infty$ to certain centered Gaussian
process $\eta(t)$ with calculable covariance. Then the sequence of statistics
$\sqrt{n} D_n$ converges in distribution to the rv   $\sup_{t\geq0} |\eta(t)|$ but currently it is impossible to find explicitly its distribution.
Hence it is reasonable to determine the critical values for statistics  $D_n$ by simulation.
\bigskip

\section{Large deviations and local efficiency of  $D_n$}

Now we obtain the logarithmic large deviation asymptotics of the
sequence of statistics (\ref{D_n_Ar}) under $H_0.$
The family of kernels $\{\Xi(X, Y, Z; t), t\geq 0\}$ is not only centered but bounded. Using the results from \cite{Niki10} on
large deviations for the supremum  of non-degenerate $U$- and $V$-statistics, we obtain the following result.
\begin{theorem}
For  $a>0$
$$
\lim_{n\to \infty} n^{-1} \ln P ( D_n >a) = - f_D(a),
$$
where the function  $f_D$ is continuous for sufficiently small $a>0,$ moreover $$
f_D(a) = (18 \delta^2)^{-1} a^2(1 + o(1)) \sim 4.966 a^2, \, \mbox{as}
\, \, a \to 0.
$$
\end{theorem}

To evaluate the efficiency, first consider the Makeham alternative with the density
$g_1(x,\theta),\theta \geq 0, x \geq 0$ given above and corresponding df $G_1(x, \theta)$. By the Glivenko-Cantelli theorem
for $U$- and $V$-statistics \cite{Jan} the limit in probability under the alternative for statistics $D_n$ is equal to
\begin{gather*}
b_1(\theta):= \sup_{t\geq 0}|b_1(t,\theta)|=
\sup_{t\geq 0}
|P_{\theta}(\max(X, Y, Z)<t)-P_{\theta}(\max(X, Y)+ \frac{Z}{3}<t)|.
\end{gather*}
It is not difficult to show that
$$b_1(t,\theta) \sim 3\theta  \int_{0}^{\infty} \xi(s; t)h_1(s)ds,
$$
where again $
h_1(s)=\frac{\partial}{\partial\theta}g_1(s,\theta)\mid _{\theta=0}$ and $\xi(s;t)$ is the projection defined above in \eqref{xi}. Hence for the Makeham alternative we have for $t\geq 0:$
\begin{gather*}
b_1(t,\theta) \sim (\frac35e^{-t}-\frac92e^{-2t}+(1+6t)e^{-3t}+3e^{-4t}-\frac{1}{10}e^{-6t})\theta
, \quad \theta \to 0.
\end{gather*}

\begin{figure}[h!]
\begin{center}
\includegraphics[scale=0.35]{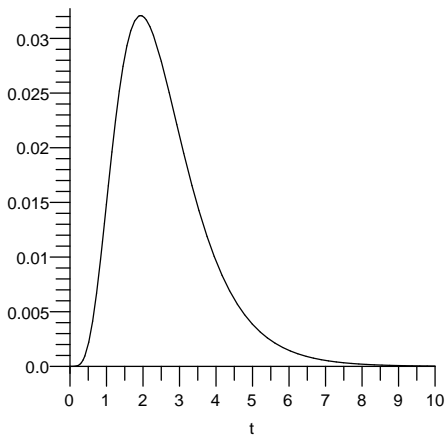}\caption{Plot of the function  $b_1(t,\theta), \mbox{ Makeham alt.}$ }
\end{center}
\end{figure}
Thus $b_1(\theta)=\sup_{t\geq 0}|b_1(t,\theta)| \sim
 0.032\,\theta,$
and it follows that the local exact slope of the sequence of statistics
$D_n$ admits the representation:
$$c_1(\theta) \sim b^2_1(\theta)/(9\delta^2) \sim 0.4599\,\theta^2, \, \theta \to 0.$$
The Kullback-Leibler "distance"\, in this case satisfies
$K_1(\theta) \sim \frac{\theta^2}{24},\, \theta \to 0$, and the local BE is
$0.123.$

Next we take the Weibull distribution, where the calculations are similar, and
the local BE is equal to $0.079.$ In the case of the Lehmann density
and the Gamma density we find that the local BE's
are 0.330 and 0.066. We collect the values of local BE in the Table 2.
\begin{table}[!hhh]\centering
\begin{tabular}{|c|c|c|c|c|}
\hline
Alternative & Makeham & Weibull & Lehmann &Gamma \\
\hline
Efficiency & 0.123 & 0.079  &  0.330 & 0.066 \\
\hline
\end{tabular}
\caption{Local Bahadur efficiencies of the statistic $D_n.$}
\end{table}

We observe that the efficiencies for the Kolmogorov-type test are lower than for the integral test. However,
it is the usual situation when testing goodness-of-fit \cite{Nik}, \cite{Rank}, \cite{Niki10}. Probably the low values of efficiencies for the Kolmogorov type test are related to the intrinsic properties of the underlying characterization of exponential law.

\bigskip
\section{Conditions of local asymptotic optimality}
\bigskip

The efficiency values of our tests for standard alternatives are far from maximal ones. Nevertheless, there exist such special alternatives (we call them "most favorable")
for which our sequences of statistics $I_n$ and $D_n$ are locally asymptotically optimal (LAO) in Bahadur sense, see general theory in \cite[Ch.6]{Nik}.
This means to describe the local structure of the
alternatives for which the given statistic has maximal potential local efficiency so that the relation
$$
c_T(\theta) \sim 2 K(\theta),\, \theta \to 0,
$$
holds, (see \cite{Nik}, \cite{NikTchir}).
Such alternatives form the domain of LAO for the given sequence of statistics.

Denote by $\cal G$
the class of densities $ g(\cdot \ ,\theta)$ with the df $G(\cdot \ ,\theta)$ which satisfy the regularity conditions
 listed below. Consider the functions
\begin{gather*}
H(x)=\frac{\partial}{\partial\theta}G(x,\theta)\mid
_{\theta=0},\quad
h(x)=\frac{\partial}{\partial\theta}g(x,\theta)\mid _{\theta=0}.
\end{gather*}

Suppose also that the following regularity conditions hold:
\begin{gather}
h(x)=H'(x), \,  x \geq 0, \quad \int_0^\infty h^2(x)e^{x}dx <  \infty  , \label{exp1}\\
\frac{\partial}{\partial\theta}\int_0^\infty x g(x,\theta)dx \mid
_{\theta=0} \ = \ \int_0^\infty x h(x)dx.\label{exp2}
\end{gather}
It is easy to show, see also \cite{NikTchir}, that under these conditions
$$ 2K(\theta)\sim \{\int_0^\infty h^2(x)e^{x}dx -(\int_0^\infty x h(x)dx)^2\} \theta^2,\, \theta \to 0.
$$

Let introduce the auxiliary function
\begin{gather}\label{h_0}
h_0(x) = h(x) - (x-1)\exp(-x)\int_0^\infty u h(u) du.
\end{gather}

First we consider the integral statistic $I_n$ with kernel
$\Psi(X_{1}, X_{2},X_{3}, X_{4})$ and the projection $\psi(x)$ from \eqref{psi} with corresponding variance $\Delta^2$ of the projection.

\begin{theorem}
Under regularity conditions \eqref{exp1}-\eqref{exp2} the alternative densities $ g(x ,\theta)$
constitute the LAO class in the class $\cal G$ for the integral statistic $I_n$ iff they have the form
$$
h(x) =g_{\theta} (x,0)= e^{-x}(C_1\psi(x)+
C_2(x-1))
$$
for some constants $C_1>0$ and $C_2 \in \mathbb{R}.$
\end{theorem}

\begin{proof}
We recall that for the integral statistic (\ref{I_n_Ar}) we have $b_I(\theta) \sim 4\theta  \int_{0}^{\infty} \psi(x)h(x)dx.$
It is straightforward that
$$
\begin{array}{ll}
\medskip
\int_0^\infty h^2(x)e^{x}dx -(\int_0^\infty x h(x)dx)^2 = \int_0^\infty h_0^2(x) e^{x} dx,\\
\int_{0}^{\infty} \psi(x)h(x)dx = \int_{0}^{\infty} \psi(x)h_0(x)dx.
\end{array}
$$

Consequently the local BE takes the form
\begin{gather*}
e^B(I)= \lim_{\theta \to 0} b_I^2(\theta) / \left(32\Delta^2 K(\theta)\right) =\\
= \left( \int_{0}^{\infty} \psi(x)h_0(x)dx\right)^2/\left(
\int_{0}^{\infty}\psi^2(x) e^{-x}dx \cdot  \int_0^\infty h_0^2(x)e^{x}dx
 \right).
\end{gather*}

The local Bahadur asymptotic optimality means that the expression
in the right-hand side is equal to 1. It follows from Cauchy-Schwarz inequality,
see also \cite{NiPe},
that it happens iff
$h_0(x)=C_1 e^{-x}\psi(x)$ for some constants $C_1>0,$
so that $h(x) = e^{-x}(C_1\psi(x)+
C_2(x-1))$ for some constants $C_1>0$ and $C_2.$
\end{proof}

The example of such alternative is the density $g(x,\theta)$ which for  small $\theta > 0$  satisfies the formula
\begin{equation*}
g(x,\theta)=e^{-x}\left(1+\theta
\left(\frac{3}{8}e^{-x}-\frac{9}{16}e^{-2x}+\frac{1}{4}e^{-3x}-\frac{1}{12}e^{-x/3}\right)\right), x \geq 0.
\end{equation*}

Now consider the Kolmogorov-type statistic (\ref{D_n_Ar}) with the family of kernels
$\Xi(X, Y, Z ;t)$ and the projection $\xi(x;t)$ from \eqref{xi} with corresponding variances $\delta^2(t)$ of these projections.
\begin{theorem}
Under regularity conditions \eqref{exp1}-\eqref{exp2} the alternative densities $ g(x ,\theta)$
form the domain of LAO in the class $\cal G$ for the statistic $D_n$ iff the function $ h(x)$ has the form
$$
h(x)=e^{-x}(C_1\xi(x; t_0)+ C_2(x-1))
$$
for
\begin{equation}
\label{tt}
t_0= \arg \max_{t\geq0} \delta^2(t)
\end{equation}
and some constants $C_1>0, \, C_2 \in \mathbb{R}.$
\end{theorem}

\begin{proof}
In this case we recall that for the statistic (\ref{D_n_Ar}) we have
$$b_D(t,\theta) \sim 3\theta  \int_{0}^{\infty} \xi(x; t)h(x)dx,
$$

Therefore the local BE is equal to
\begin{multline*}
e^B (D)= \lim_{\theta \to 0} b_D^2(\theta)/ \sup_{t\geq 0}\left(18
\delta^2(t)\right) K(\theta) =\\ = \sup_{t\geq 0}\left( \int_{0}^{\infty}
 \xi(x;t)h_0(x)dx\right)^2 / \ \sup_{t\geq 0} \left(
\int_{0}^{\infty}\xi^2 (x;t) e^{-x}dx \cdot \int_0^\infty h_0^2 e^{x} dx\right).
\end{multline*}

It follows that the sequence of statistics $D_n$ is locally optimal
iff  $h_0(x)=e^{-x}C_1\xi(x; t_0)$
for $t_0$ from (\ref{tt})
and some constants $C_1>0, \, C_2.$ Using (\ref{h_0}) we complete the proof.
\end{proof}

 The simplest example of
such alternative density  $g(x,\theta)$  for small $\theta > 0$ is given by the formula
\begin{gather*}
\label{special1}
g(x,\theta)=e^{-x} (1+\theta
\textbf{1}\{x <t_0\}\left[\frac13-e^{-t_0}+e^{-2t_0}-e^{2x-3t_0}+\frac23e^{3x-3t_0}\right]-\\
-\frac{\theta}{3} \textbf{1}\{x<3t_0\}(1-e^{t_0-x/3})^2 ), x \geq 0,
\end{gather*}
where $t_0$ is from (\ref{tt}).

\end{document}